\documentclass[]{imsart}

\RequirePackage[OT1]{fontenc}
\RequirePackage{amsthm,amsmath}
\RequirePackage{natbib}
\RequirePackage[colorlinks,citecolor=blue,urlcolor=blue]{hyperref}

\setattribute{journal}{name}{}
\usepackage{amsmath,amssymb, a4wide}
\usepackage{dsfont, color}
\usepackage{verbatim}
\usepackage{imsart}

\startlocaldefs
\numberwithin{equation}{section}
\theoremstyle{plain}
\newtheorem{theorem}{Theorem}

\newtheorem{lemma}{Lemma}

\newtheorem{proposition}{Proposition}

\endlocaldefs

\begin{document}

\begin{frontmatter}
\title{Approximation of rejective sampling inclusion probabilities and application to high order correlations}
\runtitle{Approximation of rejective sampling inclusion probabilities}
\begin{aug}
\author{\fnms{Hel\`ene} \snm{Boistard,}%\thanksref{t1,t2}
\ead[label=e1]{helene@boistard.fr}
\ead[label=u1,url]{helene@boistard.fr}}
\author{\fnms{Hendrik P.} \snm{Lopuha\"a}%\thanksref{t3}
\ead[label=e2]{h.p.lopuhaa@tudelft.nl}
\ead[label=u2,url]{h.p.lopuhaa@tudelft.nl}}
\and
\author{\fnms{Anne} \snm{Ruiz-Gazen}
\ead[label=e3]{anne.ruiz-gazen@tse-fr.eu}
\ead[label=u3]{anne.ruiz-gazen@tse-fr.eu}}
\address{Toulouse School of Economics\\
%Toulouse, 
France\\
\printead{e1,e3}\\
DIAM\\
Delft University of Technology\\
Delft, The Netherlands\\
\printead{e2}\\
}

\runauthor{Boistard, Lopuha\"a and Ruiz-Gazen}

\affiliation{Toulouse School of Economics and Delft University of Technology}

\end{aug}

\begin{abstract}
This paper is devoted to rejective sampling.
We provide an expansion of joint inclusion probabilities of any order in terms of the inclusion probabilities of order one,
extending previous results by~\cite{Hajek_1964} and~\cite{Hajek_1981} and making the remainder term more precise.
Following~\cite{Hajek_1981}, the proof is based on Edgeworth expansions.
The main result is applied to derive bounds on higher order correlations,
which are needed for the consistency and asymptotic normality of several complex estimators.
\end{abstract}



\tableofcontents
\end{frontmatter}

%\newpage
\section{Introduction}

In a finite population of size $N$, sampling without replacement with unequal inclusion probabilities
and fixed sample size is not straightforward, but
there exist several sampling designs that satisfy these properties (see \cite{BrewerHanif_1983} for a review).
Rejective sampling, which is also called maximum entropy sampling
or conditional Poisson sampling, is one possibility, introduced by \cite{Hajek_1964}.
If $n$ denotes the fixed sample size,
the $n$ units are drawn independently with
probabilities that may vary from unit to unit and the samples in which all units are not distinct are rejected.
In the particular case of equal drawing probabilities, rejective sampling coincides with simple random
sampling without replacement.
Rejective sampling with size $n$ can also be regarded as Poisson sampling conditionally on the
sample size being equal to $n$.
The unconditional Poisson design can be easily implemented by drawing $N$ independently distributed Bernoulli random
variables with
different probabilities of success, but it has the disadvantage of working with a random sample size.
The conditional Poisson design can also be interpreted as a maximum entropy sampling design for
a fixed sample size and a given set of first order inclusion probabilities.

Rejective sampling has been extensively studied in the literature.
\cite{Hajek_1964,Hajek_1981} derives an approximation of the joint
inclusion probabilities in terms of first order inclusion probabilities.
By showing that the maximum entropy design belongs to a parametric exponential family, \cite{Chenetal_1994}
give a recursive expression of the joint inclusion probabilities and propose a new algorithm.
This algorithm has been improved by \cite{Deville_2000}, who gives another expression for the joint inclusion probabilities.
Using the results in~\cite{Chenetal_1994}, \cite{Qualite_2008} proves that the variance of the well-known
unbiased Horvitz-Thompson estimator for rejective sampling is smaller than the
variance of the Hansen-Hurvitz estimator for multinomial sampling.
Several estimators of the variance of the Horvitz-Thompson estimator have also been proposed; see~\cite{MateiTille_2005}
for a comparison by means of a large simulations study. The conditional Poisson sampling scheme is not only of interest in the
survey sampling field, but also in the context of case-control studies or survival analysis~\cite{Chen_2000}.

The purpose of the present article is to generalize the result given in~\cite{Hajek_1964} and~\cite{Hajek_1981},
obtained for the first and second order inclusion probabilities of rejective sampling,
to inclusion probabilities of any order and also to provide a more precise remainder term.
The proof of our result is along the lines of the proof by~\cite{Hajek_1981} using Edgeworth expansions and
leads to approximations that are valid when $N$, $n$ and $N-n$ are large enough.
One interesting application of our result is that it enables us to show that rejective sampling
satisfies the assumptions needed for the consistency and the asymptotic normality of some complex estimators,
such as the ones defined in~\cite{BreidtOpsomer_2000},~\cite{breidt_2007},~\cite{cardot_2010}
or~\cite{Wang_2009}.
Such assumptions involve conditions on correlations up to order four,
which are difficult to check for complex sampling designs that
go beyond simple random sampling without replacement or Poisson sampling.
Our result implies that the rejective sampling design also satisfies these conditions.

In the case-control context,~\cite{Arratiaetal_2005} consider rejective sampling and
also give approximations of higher order correlations.
Their approach and the assumptions they need to derive their results are different from the ones we consider in the present paper.
Instead of using Edgeworth expansions, they consider an expansion that involves the characteristic function.
Their results are obtained using a condition, which is sufficient, but not necessary to derive our expansion. In view of this we provide an example of a rejective sampling design
that does not satisfy the condition in~\cite{Arratiaetal_2005}, but does satisfy our weaker assumption.
Moreover, Arratia \emph{et al} do not give an explicit approximation formula
for higher order inclusion probabilities in rejective sampling, whereas we do provide such an approximation,
which may be of interest in itself.

The paper is organized as follows: in Section~\ref{section:main_result} we introduce notations and  state our main result
which is Theorem \ref{theorem:main}. In Section \ref{section:application},
we apply this result and illustrate that rejective sampling satisfies
conditions on higher order correlations imposed in the recent literature to derive several asymptotic results.
Detailed proofs are provided in Section~\ref{section:proofs}.

\section{Notations and main result}
\label{section:main_result}
In this paper, we use the first description of rejective sampling by
\cite{Hajek_1981}, namely as Poisson sampling conditionally on the
sample size being equal to $n$.
Let us denote $\mathcal{U}$ as the population of size $N$.
Let $0\leq p_1,p_2, \dots, p_N\leq 1$ be a sequence of real numbers such that $p_1+p_2+\cdots+p_N=n$.
The Poisson sampling design with parameters
$p_1,p_2, \dots, p_N$ is such that for any sample $s$, the probability of $s$ is
$$
P(s)=\prod_{i\in s}p_i\prod_{i\notin s}(1-p_i).
$$
The corresponding rejective sampling design is such that the probability of a sample $s$ is
\begin{equation}
\label{def:rejective}
P_{RS}(s)
=
\begin{cases}
\displaystyle c\prod_{i\in s} p_i \prod_{i\notin s} (1-p_i) & \text{if size $s=n$},  \\
0& \text{otherwise,}
\end{cases}
\end{equation}
where $c$ is a constant such that $\sum_sP_{RS}(s)=1$.
We refer the reader to~\cite{Hajek_1981} for more details.

The inclusion probabilities of order $k$ under this sampling scheme are denoted as
\begin{equation*}
\pi_{i_1, i_2,\dots, i_k}=P_{RS}(i_1\in s, i_2\in s, \dots, i_k\in s)%\label{def:prob_inclusion}
\end{equation*}
for any $\{i_1, i_2,\dots, i_k\} \subset \{1,2, \dots, N\}$.
Our purpose is to obtain an expansion of inclusion probabilities of any order.
Theorem 7.4 in~\cite{Hajek_1981}, see also Theorem 5.2 in \cite{Hajek_1964},
provides such an expansion for inclusion probabilities of order two, i.e.,
\begin{equation}
\label{eq:theorem Hajek}
\pi_{ij}=\pi_i\pi_j\left[1-d^{-1}(1-\pi_i)(1-\pi_j)+o(d^{-1})\right],
\quad
\text{ as }
d\to\infty,
\end{equation}
uniformly in $i, j$ such that $1\leq i\neq j\leq N$,
where
\begin{equation}
\label{def:d}
d=\sum_{i=1}^Np_i(1-p_i).
\end{equation}
We will obtain an extension of~\eqref{eq:theorem Hajek}
and prove that a similar expansion holds for inclusion probabilities of higher order.

Our approach is along the lines of the method used in~\cite{Hajek_1981}.
Consider Poisson sampling with parameters $p_1, p_2,\dots, p_N$ and denote as $P$ the corresponding probability measure
on the set of samples under this sampling scheme.
For $i=1,2, \dots, N$, we denote as $I_i$ the indicator of inclusion of unit $i$,
that is
$$
I_i=\mathbf{1}(i\in s)
=
\begin{cases} 1 & \text{if } i\in s\\
0 & \text{otherwise}.
\end{cases}
$$
For every $i=1,2,\ldots,N$, the indicator $I_i$ is a Bernoulli random variable with parameter $p_i$.
Define
\begin{equation}
\label{def:K}K=\text{size } s=I_1+I_2+\cdots +I_N.
\end{equation}
Note that the expectation and the variance of $K$ satisfy $\mathbb{E}_P(K)=n$ and $\mathbb{V}_P(K)=d$.
By Bayes formula and by independence of the $I_i$'s under Poisson sampling,
the inclusion probability $\pi_{i_1,i_2,\ldots, i_k}$ can be written as
\begin{equation}
\label{eq:Bayes rule}
\begin{split}
\pi_{i_1,i_2,\ldots, i_k}
&=
P\left(I_{i_1}=I_{i_2}=\cdots=I_{i_k}=1|K=n\right)\\
&=
P(I_{i_1}=I_{i_2}=\cdots=I_{i_k}=1)\, \frac{P(K=n|I_{i_1}=I_{i_2}=\cdots=I_{i_k}=1)}{P(K=n)}\\
&=
p_{i_1}p_{i_2}\cdots p_{i_k}\, \frac{P(K=n|I_{i_1}=I_{i_2}=\cdots=I_{i_k}=1)}{P(K=n)}.
\end{split}
\end{equation}
The next step is to use Edgeworth expansions for the probabilities of $K$.
This leads to the next lemma.
\begin{lemma}
\label{lemma:probaK}
Consider Poisson sampling with parameters $p_1, p_2,\dots, p_N$, such that $p_1+p_2+\cdots+p_N=n$ with corresponding probability measure
$P$ on the set of samples. %under this sampling scheme.
Let $d$ and $K$ be defined in~\eqref{def:d} and~\eqref{def:K}, respectively.
Then, for all $A_k=\{i_1,i_2,\ldots,i_k\}\subset\{1,2,\ldots,N\}$, $k\geq 1$,
it holds that if~$d\to\infty$, then
\[
\begin{split}
P(K=n)
&=
(2\pi d)^{-1/2}
\left\{1+
c_1d^{-1}+O\left(d^{-2}\right)
\right\},\\
P(K=n|I_{i_1}=\dots=I_{i_k}=1)
&=
(2\pi d)^{-1/2}
\left\{
1+c_2d^{-1}+O\left(d^{-2}\right)
\right\},
\end{split}
\]
where
\begin{equation}
\label{def:c1}
\begin{split}
c_1
&=
\frac{1}{8}\left(1-6\overline{\overline{p(1-p)}}\right)-\frac{5}{24}\left(1-2\overline{\overline{p}}\right)^2,\\
c_2
&=
\frac12\left(B_2-(B_1-k)^2\right)-\frac12(B_1-k)\left(1-2\overline{\overline{p}}\right)+c_1,
\end{split}
\end{equation}
with
\begin{equation}
\label{ed:def B1}
\begin{split}
\overline{\overline{p}}
&=
d^{-1}\sum_{i=1}^Np_i^2(1-p_i),\\
\overline{\overline{p(1-p)}}
&=
d^{-1}\sum_{i=1}^Np_i^2(1-p_i)^2,
\end{split}
\qquad
\begin{split}
B_1
&=
\sum_{j\in A_k}p_j,\\
B_2
&=
\sum_{j\in A_k}p_j(1-p_j).
\end{split}
\end{equation}
\end{lemma}
The proof of the lemma is provided in Section~\ref{section:proofs}.
We are now in the position to formulate our main result.
\begin{theorem}
\label{theorem:main}
For $k\geq 1$, let $A_k=\{i_1,i_2,\dots, i_k\}\subset\{1,\dots, N\}$.
Under rejective sampling~(\ref{def:rejective}),
the following approximations hold as $d\to \infty$, where $d$ is defined by (\ref{def:d}).
\begin{itemize}
\item[(i)]
For all $k\geq 2$,
%\begin{eqnarray}\label{eq:theo1}
%\frac{\pi_{i_1,i_2,\ldots,i_k}}{\pi_{i_1}\pi_{i_2}\cdots \pi_{i_k}}
%-
%1
%=
%-d^{-1}\sum_{i,j\in \{i_1,i_2,\ldots, i_k\}: i<j}(1-p_i)(1-p_j)+O(d^{-2}),
%\end{eqnarray}
\begin{equation}
\label{eq:theo1}
\pi_{i_1,i_2,\ldots,i_k}=\pi_{i_1}\pi_{i_2}\cdots \pi_{i_k}\left(1-
d^{-1}\sum_{i,j\in A_k: i<j}(1-p_i)(1-p_j)+O(d^{-2})\right),
\end{equation}
where $O(d^{-2})$ holds uniformly in $i_1, i_2, \dots, i_k$.
\item[(ii)]
For all $k\geq 2$,
\begin{equation}
\label{eq:theo2}
%\frac{\pi_{i_1,i_2,\ldots,i_k}}{\pi_{i_1}\pi_{i_2}\cdots \pi_{i_k}}
%-
%1
%=
%-d^{-1}\sum_{i,j\in \{i_1,i_2,\ldots, i_k\}: i<j}(1-\pi_i)(1-\pi_j)+O(d^{-2}),
%\end{eqnarray}
\pi_{i_1,i_2,\ldots,i_k}
=\pi_{i_1}\pi_{i_2}\cdots \pi_{i_k}\left(1
-d^{-1}\sum_{i,j\in A_k: i<j}(1-\pi_i)(1-\pi_j)+O(d^{-2})\right),
\end{equation}
where $O(d^{-2})$ holds uniformly in $i_1, i_2, \dots, i_k$.
\end{itemize}
\end{theorem}
\begin{proof}
From Lemma~\ref{lemma:probaK}, we find
\[
\frac{P(K=n\mid I_{i_1}=\cdots=I_{i_k}=1)}{P(K=n)}
=
\frac{1+c_2d^{-1}+O\left(d^{-2}\right)}{1+c_1d^{-1}+O\left(d^{-2}\right)}
=
1+(c_2-c_1)d^{-1}+O(d^{-2})
\]
Together with~\eqref{eq:Bayes rule} it follows that for all $k\geq 1$,
\begin{equation}
\label{eq:expansion1}
\begin{split}
\pi_{i_1,i_2,\ldots, i_k}
&=
p_{i_1}p_{i_2}\cdots p_{i_k}
\left\{
1+(c_2-c_1)d^{-1}+O(d^{-2})
\right\}\\
&=
p_{i_1}p_{i_2}\cdots p_{i_k}
\Bigg\{
1+\frac1{2d}\sum_{j\in A_k}p_j(1-p_j)
-
\frac1{2d}\left(\sum_{j\in A_k}p_j-k\right)^2\\
&\phantom{=p_{i_1}p_{i_2}\cdots p_{i_k}
\Bigg\{
1+\frac1{2d}\sum_{j\in A_k}p_j(1-p_j)}
\,-
\frac{1-2\overline{\overline{p}}}{2d}
\left(\sum_{j\in A_k}p_j-k\right)
+O(d^{-2})
\Bigg\}.
\end{split}
\end{equation}
Applying (\ref{eq:expansion1}) to the case $k=1$,
yields that the first order inclusion probabilities satisfy
\begin{equation}
\label{eq:expansion2}
p_i=\pi_i
\left(
1
-
d^{-1}(p_i-\overline{\overline{p}})(1-p_i)+O(d^{-2})
\right),
\end{equation}
and as a consequence,
\[
p_{i_1}p_{i_2}\cdots p_{i_k}
=
\pi_{i_1}\pi_{i_2}\cdots \pi_{i_k}
\left\{
1-d^{-1}\sum_{j\in A_k}(p_j-\overline{\overline{p}})(1-p_j)+O(d^{-2})
\right\}.
\]
Combining this with (\ref{eq:expansion1}) yields
\[
\pi_{i_1,i_2,\ldots, i_k}
=
\pi_{i_1}\pi_{i_2}\cdots \pi_{i_k}
\left\{
1+ad^{-1}+O(d^{-2})
\right\}
\]
where the contribution to terms of order $d^{-1}$ is
\[
\begin{split}
a
&=
\frac1{2}\sum_{j\in A_k}p_j(1-p_j)
-
\frac1{2}\left(\sum_{j\in A_k}p_j-k\right)^2
-
\frac{1-2\overline{\overline{p}}}{2}
\left(\sum_{j\in A_k}p_j-k\right)
-\sum_{j\in A_k}(p_j-\overline{\overline{p}})(1-p_j)\\
%&=
%\frac1{2}\sum_{j\in A_k}p_j(1-p_j)
%-\frac12\left(\sum_{j\in A_k}(1-p_j)\right)^2
%+\frac12\left(\sum_{j\in A_k}(1-p_j)\right)-\overline{\overline{p}}\left(\sum_{j\in A_k}(1-p_j)\right)\\
%&\qquad
%-\sum_{j\in A_k}p_j(1-p_j)+\overline{\overline{p}}\left(\sum_{j\in A_k}(1-p_j)\right)\\
&=
-\frac1{2}\sum_{j\in A_k}p_j(1-p_j)
-\frac12\left(\sum_{j\in A_k}(1-p_j)\right)^2
+\frac12\left(\sum_{j\in A_k}(1-p_j)\right)\\
&=
\frac1{2}\sum_{j\in A_k}(1-p_j)^2
-\frac12\left(\sum_{j\in A_k}(1-p_j)\right)^2
=
-\sum_{i,j\in A_k:i<j}(1-p_i)(1-p_j).
\end{split}
\]
This proves part (i).
Part (ii) is deduced immediately from (i) and (\ref{eq:expansion2}).
\end{proof}

%==============================================================================
\section{Application: bounds on higher order correlations under rejective sampling}
\label{section:application}
Conditions on the order of higher order correlations, as $N\to\infty$,
appear at several places in the literature,
see e.g., \cite{BreidtOpsomer_2000}, \cite{breidt_2007}, \cite{cardot_2010}
or~\cite{Wang_2009}, among others.
Such conditions are used when studying asymptotic properties in survey sampling for general sampling designs,
but they are difficult to check for more complex sampling designs, that go beyond simple random sampling without replacement.
An attempt to provide simpler conditions for rejective sampling can be found in~\cite{Arratiaetal_2005}.
They formulate some sort of asymptotic stability condition on inclusion frequencies
that ensure bounds on general higher order correlations.
The purpose of the present section is to explain how Theorem~\ref{theorem:main}
can be used to establish several bounds on higher order correlations
for the rejective sampling design.
The bounds in~\cite{Arratiaetal_2005} match with the ones that we find for correlations up to order four,
which suffices for the conditions imposed in~\cite{BreidtOpsomer_2000,breidt_2007,cardot_2010,Wang_2009}.
However, in order to derive these bounds, we only need the simple requirement that
\begin{equation}
\label{eq:cond BLR}
\limsup_{N\to\infty}
\frac{N}{d}
<\infty,
\end{equation}
where $d$ is defined in~\eqref{def:d}.
Moreover, one can show that~\eqref{eq:cond BLR} is weaker than the
asymptotic stability condition in~\cite{Arratiaetal_2005} as detailed in section \ref{subsec:comp}.

Before we start a discussion on the assumptions on higher order correlations that appear for example
in~\cite{BreidtOpsomer_2000,breidt_2007,cardot_2010,Wang_2009}, first
note that~\eqref{eq:cond BLR} necessarily yields that $d\to\infty$,
which means that Theorem~\ref{theorem:main} holds.
Moreover, condition~\eqref{eq:cond BLR} has a number of additional consequences,
such as
$n\geq d\to\infty$, $N-n\geq d\to\infty$,
and
\begin{equation}
\label{eq:cons1}
\limsup_{N\to\infty}\frac{N}{n}\leq \limsup_{N\to\infty}\frac{N}{d} <\infty.
\end{equation}
A typical example of a condition on higher order correlations, is
\begin{equation}
\label{eq:cond BO CJ}
\begin{split}
\limsup_{N\to\infty}\ n \max_{(i,j)\in D_{2,N}}\left|\mathbb{E}_P(I_i-\pi_i)(I_j-\pi_j)\right|<\infty,
\end{split}
\end{equation}
where for every integer $t\geq 1$:
\begin{equation}
\label{eq:def DtN}
D_{t,N}=
\left\{
(i_1,i_2,\ldots,i_t):
i_1,i_2,\ldots,i_t\text{ are all different and each }
i_j\in\{1,2,\ldots,N\}
\right\}.
\end{equation}
Condition~\eqref{eq:cond BO CJ} is one of the assumptions in~\cite{BreidtOpsomer_2000} among others.
Since $\mathbb{E}_P(I_i-\pi_i)(I_j-\pi_j)=\pi_{ij}-\pi_i\pi_j$, condition~\eqref{eq:cond BO CJ} immediately follows
from Theorem~\ref{theorem:main} and~\eqref{eq:cons1}.

Interestingly, the simple representation of the second order correlations as a difference of second order inclusion probabilities and the product
of single order inclusion probabilities can be generalized for correlations of higher order as precised in the following lemma.
\begin{lemma}\label{lem:corr}
For any $k\geq 2$,
let $A_k=\{i_1,i_2,\ldots,i_k\}\subset\{1,2,\ldots,N\}$.
Then
\begin{equation}\label{eq:corr}
\mathbb{E}
\left[
\prod_{j=1}^k (I_{i_j}-\pi_{i_j})
\right]
=
\sum_{m=2}^k
(-1)^{k-m}
\sum_{(i_1,\ldots,i_m)\in D_{m,k}}
\left(
\pi_{i_1,\ldots,i_m}-\pi_{i_1}\cdots\pi_{i_m}
\right)
\pi_{i_{m+1}}\cdots\pi_{i_k},
\end{equation}
where $D_{m,k}$ is the set of distinct $m$-tuples in $A_k$
and $\{i_{m+1},\ldots,i_k\}=A_k\setminus \{i_1,\ldots,i_m\}$.
\end{lemma}
From this lemma, we can prove the following proposition that provides an expansion of higher
order correlations for rejective sampling.
\begin{proposition}\label{prop:corr}
Consider a rejective sampling design.
Then, for any $k\geq 3$ and any positive integers $n_j$, $j=1,2,\ldots,k$,
\begin{equation}\label{eq:rejcorr}
\mathbb{E}
\left[
\prod_{j=1}^k (I_{i_j}-\pi_{i_j})^{n_j}
\right]=O(d^{-2})
\end{equation}
as $d\to \infty$, where $d$ is defined by (\ref{def:d}).
\end{proposition}
The proofs of Lemma~\ref{lem:corr} and Proposition~\ref{prop:corr} are provided in Section~\ref{subsec:corr}.

Proposition~\ref{prop:corr} together with condition~\eqref{eq:cons1} imply that the following conditions
that appear for example in~\cite{BreidtOpsomer_2000} are satisfied:
\begin{equation}
\label{eq:cond W BO 1}
\begin{split}
&
\limsup_{N\to\infty}
\frac{N^4}{n^2}
\max_{(i,j,k,l)\in D_{4,N}}
\left|\mathbb{E}(I_{i}-\pi_{i})(I_{j}-\pi_{j})(I_{k}-\pi_{k})(I_{l}-\pi_{l})\right|<\infty\\
&\displaystyle
\limsup_{N\to\infty}
\frac{N^3}{n^2}
\max_{(i,j,k)\in D_{3,N}}\left|\mathbb{E}(I_{i}-\pi_{i})^2(I_{j}-\pi_{j})(I_{k}-\pi_{k})\right|<\infty.
\end{split}
\end{equation}
Other conditions on higher order correlations, such as
\begin{equation}
\label{eq:cond BO CJ 2}
\lim_{N\to\infty}
\max_{(i,j,k,l)\in D_{4,N}}
\left|\mathbb{E}(I_{i}I_{j}-\pi_{ij})(I_{k}I_{l}-\pi_{kl})\right|=0,
\end{equation}
that appears in~\cite{BreidtOpsomer_2000},
can be treated in the same manner.

The conditions in \cite{breidt_2007} and \cite{cardot_2010} on higher order correlations are equivalent to
the preceding ones,
whereas the conditions in \cite{Wang_2009}
differ somehow but they are all satisfied for the rejective sampling design as consequences of Proposition~\ref{prop:corr}.

%==============================================================================
\section{Proofs}\label{section:proofs}
\subsection{Proof of Lemma~\ref{lemma:probaK}}\label{subsec:lem1}
For the proof of Lemma~\ref{lemma:probaK},
we use Edgeworth expansions for probabilities of sums of independent random variables,
as given in Theorem 6.2 in~\cite{Hajek_1981}.
If $K=I_1+I_2\dots+I_N$ is a sum of independent Bernoulli random variables with
parameters $p_1,p_2, \ldots, p_N$,
and $d=\mathbb{V}(K)$.
Then, for $0\leq l\leq N$ and $m\geq 1$,
\begin{equation}
\label{eq:Petrov}
|P(K=l)-f_m(x)|=o(d^{-(m+1)/2})
\end{equation}
where $f_m(x)$ is the Edgeworth expansion of $P(K=l)$ up to order $m$, given by
\begin{equation}
\label{def:fm}
f_m(x)
=
d^{-1/2}\phi(x)\left(1+\sum_{j=1}^m P_j(x)\right),
\quad
\text{ with }
x=\frac{l-\mathbb{E}(K)}{d^{1/2}},
\end{equation}
where $\phi$ denotes the standard normal density and each $P_j$ is a linear combination of (probabilistic)
Hermite polynomials involving the cumulants of $K$.
Recall that the Hermite polynomials are defined by
\begin{equation}
\label{eq:def Hermite Polynomial}
H_k(x)=(-1)^k \text{e}^{x^2/2}\frac{\text{d}^k}{\text{d}x^k}\left[\text{e}^{-x^2/2}\right]
\end{equation}
for $k=0,1,2,\ldots$.

\begin{lemma}
\label{lemma:polynomials}
The polynomials $P_j$ in~\eqref{def:fm} can be expressed as:
\begin{equation}
\label{eq:order Hermite}
P_j(x)
=
d^{-j/2}
\sum_{\{k_m\}}
%\frac{H_{j+2r}(x)}{d^{j/2}}
H_{j+2r}(x)
\prod_{m=1}^j
\frac{1}{k_m!}
\frac{1}{((m+2)!)^{k_m}}
%\prod_{m=1}^j
\left(\frac{\kappa_{m+2}}{d}\right)^{k_m},
\end{equation}
where the sum is taken over all sets $\{k_m\}$ consisting of all non-negative integer solutions of
%\begin{equation}
%\label{eq:cond k r}
%\begin{split}
%k_1+2k_2+\cdots+jk_j&=j,\\
%k_1+k_2+\cdots+k_j&=r,
%\end{split}
%\end{equation}
\begin{equation}
\label{eq:cond k r}
k_1+2k_2+\cdots+jk_j=j,
\end{equation}
and $r$ is defined by $k_1+k_2+\cdots+k_j=r$,
and where $\kappa_m$ is the $m$-th cumulant of $K$
and
$H_{j+2r}$ is the Hermite polynomial of degree $j+2r$ as given
in~\eqref{eq:def Hermite Polynomial}.
\end{lemma}
\begin{proof}
The proof relies on the Edgeworth expansion of $P(K=l)$, e.g., see (43) in~\cite{Blinnikov_Moesner_1998},
\[
P\left(K=l\right)
=
d^{-1/2}\phi(x)
\left\{
1+\sum_{j=1}^{\infty}
d^{j/2}
\sum_{\{k_m\}}
H_{j+2r}(x)
\frac{1}{k_m!}
\prod_{m=1}^j
\left(\frac{S_{m+2}}{(m+2)!}\right)^{k_m}
\right\},
\]
where $x=(l-\mathbb{E}_P(K))d^{-1/2}$ and $S_m=\kappa_m/d^{m-1}$.
This means that
\[
P_j(x)=d^{j/2}
\sum_{\{k_m\}}
H_{j+2r}(x)
\prod_{m=1}^j
\frac{1}{k_m!}
\left(\frac{S_{m+2}}{(m+2)!}\right)^{k_m}.
\]
Note that
\[
\prod_{m=1}^j
S_{m+2}^{k_m}
=
\prod_{m=1}^j
\left(
\frac{\kappa_{m+2}}{d^{m+1}}
\right)^{k_m}
=
\prod_{m=1}^j\left(\dfrac{\kappa_{m+2}}{d}\right)^{k_m}
\prod_{m=1}^j
d^{-mk_m}=d^{-j}\prod_{m=1}^j\left(\dfrac{\kappa_{m+2}}{d}\right)^{k_m},
\]
according to~\eqref{eq:cond k r}.
This yields (\ref{eq:order Hermite}).
\end{proof}
The next lemma shows that the cumulants of the sum of independent Bernoulli variables are of the same order as the variance.
\begin{lemma}
\label{lemma:cumulants}
Let $K=I_1+I_2+\dots+I_N$ be a sum of independent Bernoulli random variables with parameters
$p_1,p_2,\ldots, p_N$.
Let $d=\mathbb{V}(K)=\sum_{i=1}^Np_i(1-p_i)$.
Then, for any positive integer $m$, we have $\kappa_m=O(d)$,
as $d\to \infty$, uniformly in $p_1,p_2, \dots, p_N$.
\end{lemma}

\begin{proof}
Recall that the cumulants of a random variable $X$ are defined as the coefficients in the expansion of the logarithm of the moment-generating function,
i.e., if
\[
g(t)=\log \mathds{E}(\text{e}^{tX})=\sum_{m=1}^\infty \kappa_m \frac{t^m}{m!},
\]
the $m$-th cumulant is $\kappa_m=g^{(m)}(0)$.
This implies that the $m$-th cumulant $\kappa_m$ of the sum of independent Bernoulli random variables is equal to the sum of the
$m$-th cumulants $e_m$ of the individual Bernoulli variables.
Moreover, we have the following recurrence relation between the cumulants of a single Bernoulli variable with parameter~$p$:
\begin{equation}
\label{eq:khatri}
e_{m+1}=p(1-p)\frac{\text{d}}{\text{d}p}e_m,
\end{equation}
see for instance, example (c) in Section 4 in~\cite{Khatri_1959}.
It is straightforward to see that $\kappa_1=p_1+p_2+\cdots+p_N$
and $\kappa_2=\sum_{i=1}^N p_i(1-p_i)=d$.
It can be proved by induction, using (\ref{eq:khatri}), that $e_m=p(1-p)R_m(p)$,
where $R_m$ is a polynomial with degree less than or equal to $m-1$ and with coefficients depending only on $m$.
Thus,
$\kappa_m=dQ_m(p)$,
where~$Q_m(p)$ is of the form
\[
Q_m(p)
=
\frac{\sum_{i=1}^Np_i(1-p_i)R_m(p_i)}{\sum_{i=1}^Np_i(1-p_i)}
\]
and is bounded uniformly in $p_1,p_2\ldots,p_N$. This proves the lemma.
%
%
%By induction, it can be proved that
%$\kappa_m=d(1-Q_m(p))$,
%where~$Q_m(p)$ is of the form
%\[
%Q_m(p)
%=
%\frac{\sum_{i=1}^Np_i(1-p_i)(b_1p_i+\cdots+b_{m-1}p_i^{m-1})}{\sum_{i=1}^Np_i(1-p_i)}
%\]
%and is bounded uniformly in $p_1,p_2\ldots,p_N$.
%This proves the lemma.
\end{proof}

\begin{proof}[Proof of Lemma~\ref{lemma:probaK}]
We use~\eqref{eq:Petrov} with $m=4$.
Because $\mathbb{E}_P(K)=n$, formula~\eqref{def:fm} is used with $x=0$.
To precise the expressions of $P_j(0)$, for $j=1,2,3,4$, we use Lemma~\ref{lemma:polynomials}.
It follows from~\eqref{eq:def Hermite Polynomial} that the Hermite polynomials
satisfy the following recurrence relationship
\begin{equation}
\label{eq:recurrence Hermite Polynomial}
H_{k+1}(x)=-\text{e}^{x^2/2}\frac{\text{d}}{\text{d}x}\left[H_k(x)\text{e}^{-x^2/2}\right].
\end{equation}
By induction it follows from~\eqref{eq:recurrence Hermite Polynomial}
that for any integer $j=0,1,\ldots$,
the Hermite polynomials $H_{2j}$ and $H_{2j+1}$ are of the form
\[
\begin{split}
H_{2j}(x)&=a_{0j}+a_{1j}x^2+\cdots+a_{jj}x^{2j},\\%\label{eq:H even}\\
H_{2j+1}(x)&=b_{1j}x+b_{2j}x^3+\cdots+b_{jj}x^{2j+1}.%\label{eq:H uneven}
\end{split}
\]
It follows that $H_{2j+1}(0)=0$, for any integer $j$.
Combining this with Lemmas~\ref{lemma:polynomials} and~\ref{lemma:cumulants}, we can see that $P_{2j+1}(0)=0$ and $P_{2j}(0)=O(d^{-j})$ for any integer $j$.
Thus,
$P_1(0)=P_3(0)=0$ and $P_4(0)=O(d^{-2})$. Moreover,
\[
P_2(0)
=\frac{H_6(0)}{2!(3!)^2}\frac{\kappa_3^2}{d^3}+\frac{H_4(0)}{4!}\frac{\kappa_4}{d^2}
=-
\frac{15}{72}\frac{\kappa_3^2}{d^3}+\frac{3}{24}\frac{\kappa_4}{d^2}.
\]
Finally, from~\eqref{eq:khatri} one can easily deduce that $\kappa_3=d(1-2\overline{\overline{p}})$
and $\kappa_4=d(1-6\overline{\overline{p(1-p)}})$.
We then obtain:
\[
\begin{split}
P(K=n)
&=
d^{-1/2}\phi(0)
\left\{1+\sum_{j=1}^4 P_j(0)+O\left(d^{-2}\right)\right\}\\
&=
(2\pi d)^{-1/2}
\left\{
1-\frac{5}{24}\left(1-2\overline{\overline{p}}\right)^2d^{-1}+\frac{1}{8}\left(1-6\overline{\overline{p(1-p)}}\right)d^{-1}
+O(d^{-2})
\right\}\\
&=
(2\pi d)^{-1/2}
\left\{1
+
c_1d^{-1}+O\left(d^{-2}\right)
\right\}.
\end{split}
\]
For the expansion of $P(K=n|I_{i_1}=\dots=I_{i_k}=1)$,
let $E_k$ denote the event $\{I_j=1,\text{ for all }j\in A_k\}$
and define the random variable $\widetilde{K}=K\mid E_k$.
Note that it can be written as the sum of independent Bernoulli's,
\[
\widetilde{K}=\sum_{j\notin A_k}I_{j}+\sum_{j\in A_k}I^*_{j}
\]
where $I^*_j=1$.% with probability $p_j^*=1$.
Thus, we can write an Edgeworth expansion for $\widetilde{K}$ as stated in~\eqref{eq:Petrov}.
Since
\begin{equation}
\label{eq:def cond cumulants}
\begin{split}
\mathbb{E}(\widetilde{K})
&=
\sum_{j\notin A_k}p_j+k=n+k-\sum_{j\in A_k}p_j=n+k-B_1,\\
\mathbb{V}(\widetilde{K})
&=
\sum_{j\notin A_k}p_j(1-p_j)=d-\sum_{j\in A_k}p_j(1-p_j)=d-B_2,
\end{split}
\end{equation}
with $\tilde{d}=d-B_2$, the expansion is as follows:
\[
P(\widetilde{K}=n)
=
\widetilde{d}^{-1/2}\phi(\widetilde{x})
\left\{1+\sum_{j=1}^4 P_j^*(\widetilde{x})\right\}+o(\widetilde{d}^{-5/2}),
\quad
\text{ with }
\widetilde{x}
=
\frac{n-\mathbb{E}(\widetilde{K})}{\widetilde{d}^{1/2}},
\]
%where $\widetilde{d}$ is defined in~\eqref{eq:def cond cumulants} and
where the $P_j^*$'s are the polynomials given in (\ref{eq:order Hermite}) corresponding to $\widetilde{K}$.

Let us first compute an expansion for $\widetilde{d}^{-1/2}\phi(\widetilde{x})$.
We start with the expansion of $\widetilde{d}^{-1/2}$:
\begin{equation}
\label{exp:sigmatilde}
\widetilde{d}^{-1/2}
=(d-B_2)^{-1/2}
=
d^{-1/2}
\left\{1+\frac12B_2d^{-1}+O(d^{-2})\right\}.
\end{equation}
Next, remark that
\begin{equation}
\label{exp:xtilde}
\widetilde{x}
=
(d-B_2)^{-1/2}(B_1-k)=
d^{-1/2}(B_1-k)
\left\{1+\frac12B_2d^{-1}+O(d^{-2})\right\},
\end{equation}
so that
$$
\phi(\widetilde{x})
=(2\pi)^{-1/2}
\left\{
1-\frac12\widetilde{x}^2+O(\widetilde{x}^4)
\right\}
=
(2\pi)^{-1/2}
\left\{
1-\frac12(B_1-k)^2d^{-1}
+O(d^{-2})
\right\}.
$$
Together with (\ref{exp:sigmatilde}), this gives
\begin{equation}
\label{exp:sigmaphi}
\widetilde{d}^{-1/2}
\phi(\widetilde{x})
=(2\pi d)^{-1/2}
\left\{
1+a_1d^{-1}+O(d^{-2})
\right\},
\end{equation}
where $a_1=(B_2-(B_1-k)^2)/2$.
Finally, we compute $P_j^*(\widetilde{x})$, for $j=1,2,3,4$.
First, let us compute the third and fourth cumulants of $\widetilde{K}$.
We have that
\[
\begin{split}
\kappa_3^*
&=
\kappa_3-\sum_{A_k}p_j(1-p_j)(1-2p_j)=\kappa_3-B_3,\\
\kappa_4^*
&=
\kappa_4-\sum_{A_k}p_j(1-p_j)(1-6p_j+6p_j^2)=\kappa_4-B_4,
\end{split}
\]
with $B_3$ and $B_4$ ad hoc constants.
Thus, by Lemmas~\ref{lemma:polynomials} and~\ref{lemma:cumulants}
with (\ref{exp:sigmatilde}) and (\ref{exp:xtilde}),
\[
P_1^*(\widetilde{x})
=
\frac{H_3(\widetilde{x})}{6}\frac{\kappa_3^*}{\widetilde{d}^{3/2}}
=
-\frac{1}{2}\frac{\kappa_3^*}{\widetilde{d}^{3/2}}
\left(\widetilde{x}+O(\widetilde{x}^3)\right)
=
-\frac12(B_1-k)\left(1-2\overline{\overline{p}}\right)d^{-1}\left(1+O(d^{-1})\right),
\]
and likewise
\[
\begin{split}
P_2^*(\widetilde{x})
&=
\frac{H_6(\widetilde{x})}{72}\frac{(\kappa_3^*)^2}{\widetilde{d}^3}
+
\frac{H_4(\widetilde{x})}{24}\frac{\kappa_4^*}{\widetilde{d}^2}
=
\left(-\frac{5}{24}\frac{(\kappa_3^*)^2}{\widetilde{d}^3}
+
\frac{1}{8}\frac{\kappa_4^*}{\widetilde{d}^2}\right)(1+O\left(\widetilde{x}^2)\right)\\
&=
\left\{
-\frac{5}{24}\left(1-2\overline{\overline{p}}\right)^2+\frac18\left(1-6\overline{\overline{p(1-p)}}\right)
\right\}d^{-1}\left(1+O(d^{-1})\right).
\end{split}
\]
Moreover, similarly to Lemma~\ref{lemma:cumulants}, one has $\kappa_m^*=O(d)$,
for any positive integer $m$.
Hence, for any integer $j$, $P_{2j}^*(\widetilde{x})=O(d^{-j})$ and $P_{2j+1}^*(\widetilde{x})=O(d^{-(j+1)})$, so that
%\[
%P_3^*(\widetilde{x})
%=
%\frac{H_9(\widetilde{x})}{1296}\frac{(\kappa_3^*)^3}{\widetilde{d}^{9/2}}
%+
%\frac{H_7(\widetilde{x})}{144}\frac{\kappa_3^*\kappa_4^*}{\widetilde{d}^{7/2}}
%+
%\frac{H_5(\widetilde{x})}{120}\frac{\kappa_5^*}{\widetilde{d}^{5/2}}
%\sim
%\frac{\widetilde{x}}{\widetilde{d}^{3/2}}
%=
%O(d^{-2})
%\]
$P_3^*(\widetilde{x})=O(d^{-2})$
and
$P_4^*(\widetilde{x})=O(d^{-2})$.
It follows that
\begin{eqnarray}\label{exp:sumPj}
1+\sum_{j=1}^4 P_j^*(\widetilde{x})=1+c_1^*d^{-1}+O(d^{-2}),
\end{eqnarray}
where
\[
\begin{split}
c_1^*
&=
-\frac12(B_1-k)\left(1-2\overline{\overline{p}}\right)
-\frac{5}{24}\left(1-2\overline{\overline{p}}\right)^2
+\frac18\left(1-6\overline{\overline{p(1-p)}}\right)\\
&=
-\frac12(B_1-k)\left(1-2\overline{\overline{p}}\right)
+c_1.
\end{split}
\]
Combining (\ref{exp:sigmaphi}) and (\ref{exp:sumPj}) proves the lemma.
\end{proof}

\subsection{Comparison with assumptions in Arratia et al.}
\label{subsec:comp}
In \cite{Arratiaetal_2005}, the following condition is used for rejective sampling.
For all $\delta\in(0,1)$, there exists $\epsilon\in(0,1)$, such that
\begin{equation}
\label{eq:cond Arratia}
\limsup_{N\to\infty}
\frac1N
\sum_{i=1}^N
\mathds{1}
\left\{
\frac{\epsilon}{1+\epsilon}<p_i<\frac1{1+\epsilon}
\right\}
\geq
1-\delta.
\end{equation}
This condition implies our condition (\ref{eq:cond BLR}), because
\[
d
=
\sum_{i=1}^N p_i(1-p_i)\\
\geq
N(1-\delta)\frac{\epsilon}{1+\epsilon}\left(1-\frac1{1+\epsilon}\right)
\geq
N\lambda
>0,
\]
where $\lambda=(1-\delta)(\epsilon/(1+\epsilon))^2\in(0,1)$.

However, our condition is weaker, in the sense that we can construct an example which
satisfies~\eqref{eq:cond BLR}, but not~\eqref{eq:cond Arratia}.
To this end, suppose that $n/N\to\gamma\in(0,1)$.
Take $\delta\in(0,1)$, such that $0<\gamma<1-\delta<1$.
Furthermore, choose $\alpha\in(0,1)$, such that
$0<\gamma<\alpha<1-\delta<1$,
and let $k=\alpha N$.
Then define
\[
p_1=\cdots=p_k=\frac{\gamma}{\alpha}
\quad\text{ and }\quad
p_{k+1}=\cdots=p_N=\delta_n=\frac{n/N-\gamma}{1-\alpha}\to 0.
\]
First note that this choice is possible  in rejective sampling, since
\[
p_1+\cdots+p_N
=
k\times\frac{\gamma}{\alpha}+(N-k)\times\delta_n
=
N\gamma+N(1-\alpha)\frac{n/N-\gamma}{1-\alpha}=n.
\]
With these probabilities, condition~\eqref{eq:cond Arratia} is not satisfied for any $\epsilon\in(0,1)$,
because for $N$ sufficiently large
$p_{k+1}=\cdots=p_N<\epsilon/(1+\epsilon)$, so that
\[
\frac1N
\sum_{i=1}^N
\mathds{1}
\left\{
\frac{\epsilon}{1+\epsilon}<p_i<\frac1{1+\epsilon}
\right\}
\leq
\frac{k}N
=
\alpha<1-\delta,
\]
whereas condition~\eqref{eq:cond BLR} is fulfilled, as
\[
\frac{d}{N}
=
\frac{n}{N}-\frac1N\sum_{i=1}^N p_i^2
=
\frac{n}{N}-\frac{k}{N}\left(\frac{\gamma}{\alpha}\right)^2-\frac{N-k}N\delta_n^2
=
\frac{n}{N}-\frac{\gamma^2}{\alpha}-(1-\alpha)\delta_n^2
\to
\gamma-\frac{\gamma^2}{\alpha}
\geq
\lambda
\]
where $\lambda=(\gamma-\gamma^2/\alpha)/2\in(0,1)$.

\subsection{Proofs of Lemma~\ref{lem:corr} and Proposition~\ref{prop:corr} }\label{subsec:corr}

\begin{proof}[Proof of Lemma~\ref{lem:corr}]
We decompose the product in the following way:
\[
\begin{split}
&\mathbb{E}\left[\prod_{j=1}^k(I_{i_j}-\pi_{i_j})\right]\\
&=
\pi_{i_1}\pi_{i_2}\ldots \pi_{i_k}(-1)^k
+
\mathbb{E}\left[
\sum_{m=1}^k\sum_{D_{m,k}}I_{i_{j_1}}I_{i_{j_2}}\dots I_{i_{j_m}}\pi_{i_{j_{m+1}}}\dots \pi_{i_{j_k}}(-1)^{k-m}\right]\\
&=
\pi_{i_1}\pi_{i_2}\ldots \pi_{i_k}(-1)^k
+
\sum_{m=1}^k\sum_{D_{m,k}}\pi_{i_{j_1}i_{j_2}\dots i_{j_m}}\pi_{i_{j_{m+1}}}\dots \pi_{i_{j_k}}(-1)^{k-m}\\
&=
\sum_{m=1}^k\sum_{D_{m,k}}
\left(\pi_{i_{j_1}i_{j_2}\dots i_{j_m}}-\pi_{i_{j_1}}\pi_{i_{j_2}}\dots \pi_{i_{j_m}}\right)\pi_{i_{j_{m+1}}}\dots \pi_{i_{j_k}}(-1)^{k-m}\\
&\qquad+
\pi_{i_1}\pi_{i_2}\ldots \pi_{i_k}(-1)^k
+
\sum_{m=1}^k\sum_{D_{m,k}}
\pi_{i_{j_1}}\pi_{i_{j_2}}\dots \pi_{i_{j_m}}\pi_{i_{j_{m+1}}}\dots \pi_{i_{j_k}}(-1)^{k-m}.
\end{split}
\]
The last two terms on the right hand side are equal to
\[
\begin{split}
&\pi_{i_1}\pi_{i_2}\ldots \pi_{i_k}(-1)^k
+
\sum_{m=1}^k\sum_{D_{m,k}}\pi_{i_1}\pi_{i_2}\ldots \pi_{i_k}(-1)^{k-m}\\
&=
\pi_{i_1}\pi_{i_2}\ldots \pi_{i_k}(-1)^k
+
\pi_{i_1}\pi_{i_2}\ldots \pi_{i_k}\sum_{m=1}^k\binom{k}{m}(-1)^{k-m}\\
&=
\pi_{i_1}\pi_{i_2}\ldots \pi_{i_k}\sum_{m=0}^k\binom{k}{m}(-1)^{k-m}
=
\pi_{i_1}\pi_{i_2}\ldots \pi_{i_k}(1-1)^k=0.
\end{split}
\]
\end{proof}

\begin{proof}[Proof of Proposition~\ref{prop:corr}]
The proof is by induction on the powers $n_j$.
We first prove that
\begin{equation}
\label{prop:step1}
\mathbb{E}
\left[
\prod_{j=1}^k (I_{i_j}-\pi_{i_j})
\right]
=
O(d^{-2}),
\end{equation}
for any $A_k=\{i_1,i_2,\ldots,i_k\}\subset\{1,2,\ldots,N\}$, with $3\leq k\leq N$
and then add an extra power one by one.
From Lemma~\ref{lem:corr}, we have that
\[
\mathbb{E}
\left[
\prod_{j=1}^k (I_{i_j}-\pi_{i_j})
\right]
=
\sum_{m=2}^k
(-1)^{k-m}
\sum_{(i_1,\ldots,i_m)\in D_{m,k}}
\left(
\pi_{i_1,\ldots,i_m}-\pi_{i_1}\cdots\pi_{i_m}
\right)
\pi_{i_{m+1}}\cdots\pi_{i_k},
\]
where $\{i_{m+1},\ldots,i_k\}=A_k\setminus \{i_1,\ldots,i_m\}$.
From Theorem 1, we have that
\[
\pi_{i_1,\ldots,i_m}-\pi_{i_1}\cdots\pi_{i_m}
=
-\pi_{i_1}\cdots\pi_{i_m}d^{-1}
\sum_{i<j}
(1-\pi_i)(1-\pi_j)+O(d^{-2}),
\]
where the sum runs over all $i<j$, such that $i,j\in\{i_1,\ldots,i_m\}$.
This means that
\[
\begin{split}
&
\mathbb{E}
\left[
\prod_{j=1}^k (I_{i_j}-\pi_{i_j})
\right]\\
&=
-d^{-1}
\pi_{i_1}\cdots\pi_{i_k}
\sum_{m=2}^k
(-1)^{k-m}
\sum_{(i_1,\ldots,i_m)\in D_{m,k}}
\sum_{i<j}
(1-\pi_i)(1-\pi_j)
+
O(d^{-2}).
\end{split}
\]
For $2\leq m\leq k$ fixed, consider the summation
\[
\sum_{(i_1,\ldots,i_m)\in D_{m,k}}
\sum_{i<j}
(1-\pi_i)(1-\pi_j).
\]
The first summation is over all possible $(i_1,\ldots,i_m)\in D_{m,k}$, which are all possible combinations of $m$ different indices from
$A_k=\{i_1,i_2,\ldots,i_k\}$.
From each such combination $i_1,\ldots,i_m$,
the second summation picks two different indices $i<j$ from the set $\{i_1,\ldots,i_m\}$.
This means that any combination of $(1-\pi_i)(1-\pi_j)$, with $\{i,j\}\subset A_k$ is possible.
In fact, each such combination will appear several times and we only have to count how many times.
Well, for a fixed combination $(i,j)$, from the $k$ possibilities $A_k$, we  need to pick $i$ and $j$,
and for the $m-2$ remaining choices there are $k-2$ possibilities left.
We conclude that each term $(1-\pi_i)(1-\pi_j)$, with $\{i,j\}\subset A_k$, appears $\binom{k-2}{m-2}$ times.
Moreover, this holds for any $m=2,3,\ldots,k$.
This means that
\[
\begin{split}
&
\sum_{m=2}^k
(-1)^{k-m}
\sum_{(i_1,\ldots,i_m)\in D_{m,k}}
\sum_{i<j}
(1-\pi_i)(1-\pi_j)\\
&=
\sum_{\{i,j\}\subset A_k}
(1-\pi_i)(1-\pi_j)
\sum_{m=2}^k
(-1)^{k-m}
\binom{k-2}{m-2},
\end{split}
\]
where
\[
\sum_{m=2}^k
(-1)^{k-m}
\binom{k-2}{m-2}
=
\sum_{n=0}^{k-2}
(-1)^{k-2-n}
\binom{k-2}n
=
(1-1)^{k-2}=0.
\]
We conclude that the coefficient of $d^{-1}$ is zero,
which proves (\ref{prop:step1}).

Next, suppose that the expectation is of order $O(d^{-2})$ for all powers $1\leq m_j\leq n_j$,
and consider
\[
\mathbb{E}
\left[
(I_{i_1}-\pi_{i_1})^{n_1+1}(I_{i_2}-\pi_{i_2})^{n_2}\cdots(I_{i_k}-\pi_{i_k})^{n_k}
\right].
\]
This can be written as
\[
\begin{split}
&
\mathbb{E}
\left[
I_{i_1}(I_{i_1}-\pi_{i_1})^{n_1}(I_{i_2}-\pi_{i_2})^{n_2}\cdots(I_{i_k}-\pi_{i_k})^{n_k}
\right]\\
&\qquad-
\pi_{i_1}
\mathbb{E}
\left[
(I_{i_1}-\pi_{i_1})^{n_1}(I_{i_2}-\pi_{i_2})^{n_2}\cdots(I_{i_k}-\pi_{i_k})^{n_k}
\right]\\
&=
\mathbb{E}
\left[
I_{i_1}(I_{i_1}-\pi_{i_1})^{n_1}(I_{i_2}-\pi_{i_2})^{n_2}\cdots(I_{i_k}-\pi_{i_k})^{n_k}
\right]
+
O(d^{-2})
\end{split}\]
according to the induction hypothesis.
Next, write
\[
\begin{split}
I_{i_1}(I_{i_1}-\pi_{i_1})^{n_1}
&=
(1-\pi_{i_1})
I_{i_1}
(I_{i_1}-\pi_{i_1})^{n_1-1}\\
&=
(1-\pi_{i_1})
(I_{i_1}-\pi_{i_1})^{n_1}
+
(1-\pi_{i_1})\pi_{i_1}
(I_{i_1}-\pi_{i_1})^{n_1-1}.
\end{split}
\]
When we insert this, we find
\[
\begin{split}
&\mathbb{E}
\left[
(I_{i_1}-\pi_{i_1})^{n_1+1}(I_{i_2}-\pi_{i_2})^{n_2}\cdots(I_{i_k}-\pi_{i_k})^{n_k}
\right]\\
&=
(1-\pi_{i_1})
\mathbb{E}
\left[
(I_{i_1}-\pi_{i_1})^{n_1}(I_{i_2}-\pi_{i_2})^{n_2}\cdots(I_{i_k}-\pi_{i_k})^{n_k}
\right]\\
&\qquad+
(1-\pi_{i_1})\pi_{i_1}
\mathbb{E}
\left[
(I_{i_1}-\pi_{i_1})^{n_1-1}(I_{i_2}-\pi_{i_2})^{n_2}\cdots(I_{i_k}-\pi_{i_k})^{n_k}
\right]
+
O(d^{-2})\\
&=
O(d^{-2})
\end{split}\]
by applying the induction hypothesis.

\end{proof}

\paragraph{Acknowledgements.}
The authors want to thank Guillaume Chauvet for helpful discussions.

\bibliography{rejective_bib}

\end{document}